\documentclass{amsart}

\usepackage[T1]{fontenc}
\usepackage[utf8]{inputenc}
\usepackage[english]{babel}
\usepackage{amssymb,mathtools}
\usepackage{graphicx}
\usepackage{booktabs}
\usepackage{longtable}
\usepackage{array}
\usepackage{float}
\usepackage[section]{placeins}
\usepackage{algorithm2e}
\usepackage{enumitem}
\usepackage{needspace}
\usepackage{hyperref}
\usepackage[nameinlink,noabbrev]{cleveref}

\hypersetup{
  hidelinks,
  pdftitle={An Incremental Approximation to the Prime-Counting Function via Polygonal Numbers},
  pdfauthor={Artur Samojluk and Artur Siemaszko},
  pdfsubject={Version V13.7; error asymptotics and a logarithmic-mean criterion for the Riemann hypothesis}
}

\allowdisplaybreaks[1]
\numberwithin{equation}{section}
\newtheorem{theorem}{Theorem}[section]
\newtheorem{lemma}[theorem]{Lemma}
\newtheorem{proposition}[theorem]{Proposition}
\newtheorem{corollary}[theorem]{Corollary}

\theoremstyle{definition}
\newtheorem{definition}[theorem]{Definition}
\theoremstyle{remark}
\newtheorem{remark}[theorem]{Remark}

\newcommand{\C}{\mathbb C}
\newcommand{\Li}{\operatorname{Li}}
\newcommand{\dd}{\,\mathrm d}
\newcommand{\Mcal}{\mathcal M}
\newcommand{\bigO}{\mathcal O}

\title[Incremental prime-counting approximation]{An incremental approximation to the prime-counting function via polygonal numbers\\[0.5ex]
  \normalfont\small
Error asymptotics and a logarithmic-mean criterion for the Riemann hypothesis}
\author[A. Samojluk]{Artur Samojluk}
\address{University of Warmia and Mazury in Olsztyn, Poland}
\email{artur.samojluk@uwm.edu.pl}
\author[A. Siemaszko]{Artur Siemaszko}
\address{University of Warmia and Mazury in Olsztyn, Poland}
\email{artur@uwm.edu.pl}
\date{Version V13.7 -- 30 September 2026}

\crefname{externaltheorem}{external theorem}{external theorems}
\Crefname{externaltheorem}{External theorem}{External theorems}
\crefname{appendix}{appendix}{appendices}
\Crefname{appendix}{Appendix}{Appendices}
\begin{document}

\begin{abstract}
We construct an incremental prime-counting approximation from corrected right-endpoint sums on polygonal-number partitions. For each fixed polygonal family, retaining completed contributions gives $\mathcal O(1)$ numerical operations per consecutive node and $\mathcal O(\sqrt X)$ operations for the full traversal through $X$, in the unit-cost model. The scheme uses a fixed number of working numerical words rather than repeated summation. Asymptotic formulas for the discretization error and terminal scale yield an unconditional relation with the classical normalized prime-counting error. We prove that the logarithmic mean of the normalized estimator error converges to $1/\sqrt2$ if and only if the Riemann hypothesis holds, and give an independent Mellin proof of the converse. Under the hypothesis, this constant uniquely removes the limiting signed mean bias. Computations at powers of ten through $10^{19}$ with coefficient $0.7071$ show accuracy comparable to the Riemann approximation at the sampled arguments.
\end{abstract}
\maketitle

\section{Introduction}\label{sec:introduction}

Repeated approximation of the prime-counting function calls for reuse of
previous calculations. We study an estimator that retains the contributions
of completed cells in a polygonal-number partition and adds only the next
local contribution. Its computational advantage is this incremental
structure, coupled with a rigorous analysis of the approximation error.

Let $\pi(x)=\#\{p\leq x:p\text{ is prime}\}$. The prime number theorem gives
\begin{equation}\label{eq:PNT}
\pi(x)\sim\Li(x)\sim\frac{x}{\log x},\qquad x\to\infty;
\end{equation}
see \cite[Chapter~III]{Titchmarsh}. We use natural logarithms and the offset convention
\begin{equation}\label{eq:Li-def}
\Li(x)=\int_2^x\frac{\dd t}{\log t},\qquad x>1,
\end{equation}
so $\Li(2)=0$. The unshifted convention is related by $C_*:=\operatorname{Ei}(\log2)$, $\operatorname{Ei}(\log x)=\Li(x)+C_*$; see \cite[Eqs.~6.2.8 and~6.6.1]{DLMF}.

We use right-endpoint sums on the classical polygonal numbers
$P_\ell(n)=((\ell-2)n^2-(\ell-4)n)/2$, with a fixed integer $\ell\geq3$,
and a single terminal correction. Triangular numbers correspond to
$\ell=3$ and squares to $\ell=4$ \cite{BrownDunnHarrington}. The triangular
case was inspired by Samojluk's work on triangular numbers and factorization
\cite{SamojlukTriangular}; the estimator and proofs here do not depend on
that factorization algorithm. The partition is used as a quadrature rule,
not as an assumption about primes in individual cells. In particular, the
algorithm approximates cumulative prime counts; it does not determine the
exact number of primes in each cell.

Write the Samojluk--Siemaszko (S--S) family as
$S_{\ell,c}=\widehat S_\ell-cq_\ell$, where $\widehat S_\ell$ is the
right-endpoint sum with a boundary adjustment and $q_\ell$ is a terminal
scale. The recurrences
\eqref{eq:incremental-recurrence}--\eqref{eq:incremental-corrected}
reuse the accumulator and the last local term: advancing to the next node
requires one new contribution and one new logarithm. Thus the traversal
is linear in the number of generated estimates, rather than a repeated
summation of all earlier contributions. An isolated cold start still
requires forming the accumulated sum, and a jump across $r$ cells requires
$r$ updates. Section~\ref{sec:algorithm} gives the operation counts and
separates them from precision-dependent bit costs.

The comparison with the Riemann approximation $R(x)$ concerns a specified
evaluation strategy. Once its coefficients have been prepared, summing an
$m$-term Gram truncation independently at each of $N$ requested nodes
requires $Nm$ term updates; the S--S traversal adds one local contribution
per node. This avoids repeated global summation in that sequential task.
The truncation length and the errors relative to $\pi$ must also be taken
into account: the comparison is neither a lower bound for computing $R$
nor a runtime benchmark at matched accuracy, and does not exclude reuse
or accelerated evaluation of $R$.

The analytic part quantifies what this computational construction
approximates. For every fixed $\ell$, we prove
\[
q_\ell(x)\sim\frac{\sqrt{2x}}{\log x},\qquad
\Li(x)-\widehat S_\ell(x)
=\frac{\sqrt{2(\ell-2)x}}{\log^2x}
 +\bigO_\ell\!\left(\frac{\sqrt x}{\log^3x}\right).
\]
These estimates separate deterministic discretization error from the
prime-counting error and give an unconditional transfer between their
normalized logarithmic means. The classical mean under the Riemann
hypothesis (RH) identifies $1/\sqrt2$ as the unique constant correction
with zero limiting signed logarithmic-mean bias. This is not a claim of
pointwise or finite-sample absolute-error optimality.

A direct Mellin argument proves the converse: even boundedness of the
normalized mean implies RH. The underlying averaged-error criterion also
follows from Johnston's earlier weighted-integral theorem
\cite[Lemma~2.8]{Johnston2023}; we do not claim priority for that criterion.
Our converse proof does not use his oscillation theorem or require a
choice of logarithmic branches. The contribution is the independent
argument and its transfer to this estimator, not an unconditional proof
of the averaging condition or of RH.

The definitions and incremental implementation are given in
\cref{sec:partition,sec:algorithm}. Section~\ref{sec:coefficient} proves
the error estimates and the transfer of means. Section~\ref{sec:RH}
quotes the classical forward theorem and proves the converse, the RH
criterion, and the centering result.
Section~\ref{sec:numerics} compares the estimator with $R(x)$ analytically
and numerically, using exactly $c_T=0.7071$ at powers of ten through
$10^{19}$. The boundary-convention
comparison and the separate derivation from Johnston are retained with
full details in \cref{app:boundary,subsubsec:Johnston-reduction}.

\section{Polygonal partitions and the estimator}\label{sec:partition}

\subsection{Polygonal nodes, cell lengths, and inverse indices}

Fix an integer $\ell\geq3$. The $n$th $\ell$-gonal number is
\begin{equation}\label{eq:P-def}
P_\ell(n)
 =\frac{(\ell-2)n^2-(\ell-4)n}{2}
 =\frac{n\bigl((\ell-2)(n-1)+2\bigr)}2,
\qquad n\geq1;
\end{equation}
see \cite[Section~1]{BrownDunnHarrington}. The parameter $\ell$ counts sides; the row increment is $\ell-2$.

The corresponding cell length is
\begin{equation}\label{eq:h-def}
h_n^{(\ell)}:=P_\ell(n+1)-P_\ell(n)=1+(\ell-2)n.
\end{equation}
For example,
\begin{align*}
P_3(n)&=T_n=\frac{n(n+1)}2 &&\text{(triangular numbers)},\\
P_4(n)&=n^2 &&\text{(square numbers)},\\
P_5(n)&=\frac{n(3n-1)}2 &&\text{(pentagonal numbers)},\\
P_6(n)&=n(2n-1) &&\text{(hexagonal numbers)}.
\end{align*}
We retain $T_n=P_3(n)$ only as the familiar abbreviation in the
triangular implementation and numerical tables.

The same formula is obtained by summing rows of length
\begin{equation}\label{eq:rows}
r_n^{(\ell)}=1+(n-1)(\ell-2),\qquad
P_\ell(n)=\sum_{j=1}^n r_j^{(\ell)},\qquad h_n^{(\ell)}=r_{n+1}^{(\ell)}.
\end{equation}

The positive real solution of $P_\ell(t)=x$ is
\begin{equation}\label{eq:nu-def}
\nu_\ell(x)
 :=\frac{\sqrt{8(\ell-2)x+(\ell-4)^2}+\ell-4}{2(\ell-2)}.
\end{equation}
Accordingly, define the floor and ceiling indices
\begin{equation}\label{eq:mn-def}
m_x^{(\ell)}=\lfloor\nu_\ell(x)\rfloor,
\qquad n_x^{(\ell)}=\lceil\nu_\ell(x)\rceil.
\end{equation}
Thus $P_\ell\!\left(m_x^{(\ell)}\right)\leq x \leq P_\ell\!\left(n_x^{(\ell)}\right)$, with equality of the two indices when $x$ is a partition node.

The number of partition nodes not exceeding $x$ is
\begin{equation}\label{eq:node-density}
N_\ell(x):=\#\{n\geq1:P_\ell(n)\leq x\}
 =m_x^{(\ell)}=\sqrt{\frac{2x}{\ell-2}}+\bigO_\ell(1).
\end{equation}
Consequently $N_\ell(x)/N_3(x)\to1/\sqrt{\ell-2}$. This compares
the \emph{numbers of nodes}, not containment of one sequence in another.
At the same argument the terminal cell length is asymptotic to
$\sqrt{2(\ell-2)x}$, which motivates dividing its contribution by
$\sqrt{\ell-2}$ in the normalization below.

\subsection{The ceiling estimator and its natural scale}

For $j\geq1$, set
\begin{equation}\label{eq:s-def}
s_j^{(\ell)}
 :=\frac{h_j^{(\ell)}}{\log P_\ell(j+1)},
\qquad
\Sigma_n^{(\ell)}:=\sum_{j=1}^{n-1}s_j^{(\ell)}.
\end{equation}
For every $x\geq2$, the ceiling estimator is
\begin{equation}\label{eq:hatS-def}
\widehat S_\ell(x)
 :=\Sigma_{n_x^{(\ell)}}^{(\ell)}
   -\frac{P_\ell\!\left(n_x^{(\ell)}\right)-x}
          {\log P_\ell\!\left(n_x^{(\ell)}\right)}.
\end{equation}
If $n=n_x^{(\ell)}$ and $P_\ell(n-1)<x\leq P_\ell(n)$, then
\begin{equation}\label{eq:hatS-cell}
\widehat S_\ell(x)
 =\sum_{j=1}^{n-2}\frac{h_j^{(\ell)}}{\log P_\ell(j+1)}
  +\frac{x-P_\ell(n-1)}{\log P_\ell(n)}.
\end{equation}

The terminal contribution determines the normalization
\begin{equation}\label{eq:q-def}
q_\ell(x)
 :=\frac{s_{n_x^{(\ell)}-1}^{(\ell)}}{\sqrt{\ell-2}}
 =\frac{h_{n_x^{(\ell)}-1}^{(\ell)}}
        {\sqrt{\ell-2}\,\log P_\ell\!\left(n_x^{(\ell)}\right)}.
\end{equation}
Both definitions are valid on the whole interval $[2,\infty)$: since
$P_\ell(1)=1$ and $P_\ell(2)=\ell>1$, one always has
$n_x^{(\ell)}\geq2$ and $q_\ell(x)>0$. In particular, for
$2\leq x\leq P_\ell(2)$,
\begin{equation}\label{eq:initial-cell}
\widehat S_\ell(x)=\frac{x-1}{\log P_\ell(2)},\qquad
q_\ell(x)=\frac{\ell-1}{\sqrt{\ell-2}\,\log P_\ell(2)}.
\end{equation}
These formulas specify the compact initial interval needed in all subsequent
logarithmic averages. They do not require an integral starting at the
singular point $t=1$.

At a node $x=P_\ell(n)$, \eqref{eq:hatS-def} gives
$\widehat S_\ell(x)=\Sigma_n^{(\ell)}$. The formula in the next cell
has the same right limit, so $\widehat S_\ell$ is continuous.
The ceiling convention fixes $q_\ell$ at the nodes; its possible jumps do
not affect any integral. All normalized errors below are locally integrable.

We compare the normalized estimator error
\begin{equation}\label{eq:Z-def}
Z_\ell(x):=\frac{\widehat S_\ell(x)-\pi(x)}{q_\ell(x)}
\end{equation}
with the normalized prime-number-theorem error in the offset convention
\begin{equation}\label{eq:E-def}
E(x):=\frac{\log x}{\sqrt{x}}\bigl(\Li(x)-\pi(x)\bigr).
\end{equation}

\subsection{The correction family}

\begin{definition}[Corrected S--S estimator]\label[definition]{def:S-family}
For a fixed coefficient $c\in\mathbb R$, define
\begin{equation}\label{eq:S-family-def}
S_{\ell,c}(x)
 :=\widehat S_\ell(x)-c\,q_\ell(x)
 =\widehat S_\ell(x)-\frac{c}{\sqrt{\ell-2}}\,
   s_{n_x^{(\ell)}-1}^{(\ell)}.
\end{equation}
\end{definition}

Finite-range coefficients may be selected numerically, and the available experiments indicate an effective region near $0.7$. For the computations reported in \cref{sec:numerics}, we fix
\begin{equation}\label{eq:c-num}
c_T:=0.7071.
\end{equation}
This is the value of $1/\sqrt2=0.707106781\ldots$ rounded to four decimal places. The numerical results should therefore be read as the performance of the explicitly specified coefficient $c_T$, not as a new proof that this decimal value uniquely minimizes a finite-sample loss function.

For the triangular partition $\ell=3$, define the estimator used in the numerical tables by
\begin{equation}\label{eq:ST-def}
\widetilde S_T(x):=S_{3,c_T}(x)=S_{3,0.7071}(x).
\end{equation}

The theoretical coefficient defines
\begin{equation}\label{eq:Sth-def}
S_{\ell,\mathrm{th}}(x)
 :=S_{\ell,1/\sqrt2}(x)
 =\widehat S_\ell(x)-\frac1{\sqrt2}q_\ell(x)
 =\widehat S_\ell(x)-\frac{s_{n_x^{(\ell)}-1}^{(\ell)}}{\sqrt{2(\ell-2)}}.
\end{equation}

The exact coefficient and its decimal approximation are distinguished throughout;
\cref{eq:coefficient-change} quantifies their difference in the numerical section.

\section{Incremental evaluation and computational cost}\label{sec:algorithm}

At consecutive partition nodes, the S--S estimator has the exact recurrences,
for $n\geq2$,
\begin{align}
\widehat S_\ell(P_\ell(n+1))
 &=\widehat S_\ell(P_\ell(n))+s_n^{(\ell)},
 \label{eq:incremental-recurrence}\\
S_{\ell,c}(P_\ell(n+1))
 &=S_{\ell,c}(P_\ell(n))+s_n^{(\ell)}
   -\frac{c}{\sqrt{\ell-2}}\bigl(s_n^{(\ell)}-s_{n-1}^{(\ell)}\bigr).
 \label{eq:incremental-corrected}
\end{align}
The second identity follows by substituting
$q_\ell(P_\ell(n))=s_{n-1}^{(\ell)}/\sqrt{\ell-2}$ into the first.
Retaining the accumulator and the last local term makes the next update
constant-cost in the unit-cost arithmetic model.

For a cold start, compute $n=n_x^{(\ell)}$, accumulate the $n-1$ terms
in \eqref{eq:s-def}, and apply the boundary adjustment and correction.
The number of terms is $\bigO_\ell(\sqrt x)$. The stored state consists
of the index, the accumulator, and the most recent local term. A new node
requires a bounded number of arithmetic operations and one new logarithm.
Thus the operation count per consecutive node is $\bigO(1)$ in the model
that assigns unit cost to these numerical operations. No amortization over
several nodes is needed for this bound.

\begin{remark}[Scope of the computational cost]\label[remark]{rem:complexity}
The state uses $\bigO(1)$ numerical words, not a constant number of bits.
Integer arithmetic, high-precision logarithms, and accumulation have costs
that depend on argument size and precision. A jump across $r$ cells needs
$r$ direct updates. The cold-start count is not a lower bound for every
possible way to evaluate the same finite sum, and the recurrence does not
exclude accelerated evaluation methods for $R(x)$.
\end{remark}

For integer inputs the triangular index can be computed without a
floating-point square root. If $u=\lfloor\sqrt{8x+1}\rfloor$ and
$m=\lfloor(u-1)/2\rfloor$, then $T_m\leq x<T_{m+1}$. Consequently the
ceiling index is $m$ when $T_m=x$ and $m+1$ otherwise. This follows either
from the exact inverse or from $(2m+1)^2\leq8x+1<(2m+3)^2$.
Here $\operatorname{isqrt}$ denotes the integer square root. For a general
polygonal family, the analogous candidate can be checked using the exact
inequalities $P_\ell(n-1)<x\leq P_\ell(n)$.

\begin{algorithm}[H]
\DontPrintSemicolon
\KwIn{An integer $x\geq2$ and a correction coefficient $c\in\mathbb R$}
\KwOut{$\widehat S_3(x)$ and $S_{3,c}(x)$}
$u\leftarrow\operatorname{isqrt}(8x+1)$\;
$n\leftarrow\lfloor(u-1)/2\rfloor$\;
\If{$n(n+1)/2<x$}{$n\leftarrow n+1$\;}
$\Sigma\leftarrow0$\;
\For{$j\leftarrow1$ \KwTo $n-1$}{
  $T\leftarrow(j+1)(j+2)/2$\;
  $\Sigma\leftarrow\Sigma+(j+1)/\log T$\;
}
$T_n\leftarrow n(n+1)/2$\;
$\widehat S_3(x)\leftarrow\Sigma-(T_n-x)/\log T_n$\;
$q_3(x)\leftarrow n/\log T_n$\;
$S_{3,c}(x)\leftarrow\widehat S_3(x)-c\,q_3(x)$\;
\caption{Cold-start evaluation of the triangular S--S estimator at an integer argument.}
\label{alg:triangular}
\end{algorithm}

\section{Discretization error and unconditional transfer}\label{sec:coefficient}

Throughout this section the integer $\ell\geq3$ is fixed. Constants may
depend on $\ell$, but not on $x$ or its position within a cell. No estimate
is claimed uniformly when $\ell$ grows with $x$. We first estimate $q_\ell$,
then sum the complete and partial cell errors, and finally compare the
resulting normalized error with $E$. Only the last simplification uses
a prime-number-theorem remainder; the quadrature analysis is deterministic.

\subsection{Asymptotics of the denominator \texorpdfstring{$q_\ell(x)$}{q-ell(x)}}

\begin{lemma}[Natural terminal scale]\label[lemma]{lem:q-asymptotic}
As $x\to\infty$,
\begin{equation}\label{eq:q-asymptotic}
q_\ell(x)
 =\frac{\sqrt{2x}}{\log x}
  \left(1+\bigO_\ell(x^{-1/2})\right).
\end{equation}
In particular, $q_\ell(x)\sim\sqrt{2x}/\log x$.
\end{lemma}

\begin{proof}
Let $n=n_x^{(\ell)}$. The exact inverse \eqref{eq:nu-def} gives
\[
\nu_\ell(x)=\sqrt{\frac{2x}{\ell-2}}
 \sqrt{1+\frac{(\ell-4)^2}{8(\ell-2)x}}+\frac{\ell-4}{2(\ell-2)}
 =\sqrt{\frac{2x}{\ell-2}}+\frac{\ell-4}{2(\ell-2)}
   +\bigO_\ell(x^{-1/2}).
\]
Here $\sqrt{1+u}-1=u/(\sqrt{1+u}+1)=\bigO(u)$ for $u\geq0$.
Since $0\leq\lceil\nu_\ell(x)\rceil-\nu_\ell(x)<1$, it follows that
\begin{equation}\label{eq:n-asymptotic}
n=\sqrt{\frac{2x}{\ell-2}}+\bigO_\ell(1).
\end{equation}
Therefore
\begin{equation}\label{eq:h-terminal}
\frac{h_{n-1}^{(\ell)}}{\sqrt{\ell-2}}
 =\frac{(\ell-2)n+3-\ell}{\sqrt{\ell-2}}
 =\sqrt{2x}+\bigO_\ell(1).
\end{equation}
Moreover, the ceiling convention gives $0\leq P_\ell(n)-x<h_{n-1}^{(\ell)}=\bigO_\ell(\sqrt{x})$, so the mean-value theorem gives
\begin{equation}\label{eq:log-terminal}
\log P_\ell(n)=\log x+\bigO_\ell(x^{-1/2}).
\end{equation}
More explicitly,
\begin{align*}
\frac{h_{n-1}^{(\ell)}}{\sqrt{\ell-2}}
 &=\sqrt{2x}(1+\alpha_x),
 &\alpha_x&=\bigO_\ell(x^{-1/2}),\\
\log P_\ell(n)&=\log x(1+\beta_x),
 &\beta_x&=\bigO_\ell\!\left(\frac1{\sqrt{x}\log x}\right).
\end{align*}
For all sufficiently large $x$, $|\alpha_x|,|\beta_x|\leq1/2$.
The identity $(1+u)^{-1}=1-u/(1+u)$ gives
\[
q_\ell(x)=\frac{\sqrt{2x}}{\log x}\frac{1+\alpha_x}{1+\beta_x}
 =\frac{\sqrt{2x}}{\log x}(1+\bigO_\ell(x^{-1/2})),
\]
and also
\begin{equation}\label{eq:q-reciprocal}
\frac1{q_\ell(x)}=\frac{\log x}{\sqrt{2x}}
  (1+\bigO_\ell(x^{-1/2})).
\end{equation}
This proves the assertion and its reciprocal form.
\end{proof}

\subsection{Exact decomposition of the discretization error}

Let
\begin{equation}\label{eq:D-def}
D_\ell(x):=\Li(x)-\widehat S_\ell(x).
\end{equation}
For $j\geq2$, define the complete-cell error
\begin{equation}\label{eq:dj-def}
d_j^{(\ell)}
 :=\int_{P_\ell(j)}^{P_\ell(j+1)}\frac{\dd t}{\log t}
   -\frac{h_j^{(\ell)}}{\log P_\ell(j+1)}.
\end{equation}
For $x>P_\ell(2)$, let $n=n_x^{(\ell)}\geq3$. If $P_\ell(n-1)<x\leq P_\ell(n)$, define the terminal partial-cell error
\begin{equation}\label{eq:partial-d-def}
d_{n-1}^{(\ell)}(x)
 :=\int_{P_\ell(n-1)}^x\frac{\dd t}{\log t}
   -\frac{x-P_\ell(n-1)}{\log P_\ell(n)}.
\end{equation}
Finally, put
\begin{equation}\label{eq:kappa-def}
\kappa_\ell
 :=\Li(\ell)-\frac{\ell-1}{\log\ell}.
\end{equation}
Then, for $x>P_\ell(2)$ and $n=n_x^{(\ell)}\geq3$,
\begin{equation}\label{eq:D-exact}
D_\ell(x)
 =\kappa_\ell+\sum_{j=2}^{n-2}d_j^{(\ell)}+d_{n-1}^{(\ell)}(x),
\end{equation}
where an empty sum is zero. To verify the identity, write
\[
\Li(x)=\Li(P_\ell(2))
 +\sum_{j=2}^{n-2}\int_{P_\ell(j)}^{P_\ell(j+1)}\frac{\dd t}{\log t}
 +\int_{P_\ell(n-1)}^x\frac{\dd t}{\log t},
\]
and subtract \eqref{eq:hatS-cell}. The first local term, with $j=1$,
is absorbed into $\kappa_\ell$; no improper integral over $[1,P_\ell(2)]$
is introduced. At $x=P_\ell(2)$ the correct formula is simply
$D_\ell(x)=\kappa_\ell$. On $[2,P_\ell(2)]$ one uses
\eqref{eq:initial-cell}; this changes no large-$x$ asymptotic statement.

\subsection{Two-sided cell estimates}

Nonnegativity of the cell errors does not determine their size. The lower
and upper bounds in the following lemma are both needed to identify the
main term of the accumulated error.

\begin{lemma}[Exact cell identity and two-sided bounds]\label[lemma]{lem:cell-error}
For every $j\geq2$,
\begin{equation}\label{eq:cell-identity}
d_j^{(\ell)}
 =\int_{P_\ell(j)}^{P_\ell(j+1)}
   \frac{u-P_\ell(j)}{u\log^2u}\,\dd u.
\end{equation}
Consequently,
\begin{equation}\label{eq:cell-two-sided}
\frac{\bigl(h_j^{(\ell)}\bigr)^2}
     {2P_\ell(j+1)\log^2P_\ell(j+1)}
\leq d_j^{(\ell)}\leq
\frac{\bigl(h_j^{(\ell)}\bigr)^2}
     {2P_\ell(j)\log^2P_\ell(j)}.
\end{equation}
In particular,
\begin{equation}\label{eq:dj-asymptotic}
d_j^{(\ell)}
 =\frac{\ell-2}{\log^2P_\ell(j+1)}
  +\bigO_\ell\!\left(\frac1{j\log^2(j+1)}\right).
\end{equation}
\end{lemma}

\begin{proof}
Put $a=P_\ell(j)>1$ and $b=P_\ell(j+1)$. Integration by
parts, using $(1/\log u)'=-1/(u\log^2u)$, gives
\[
\int_a^b\frac{\dd u}{\log u}
 =\left[\frac{u-a}{\log u}\right]_a^b
   +\int_a^b\frac{u-a}{u\log^2u}\,\dd u
 =\frac{b-a}{\log b}+\int_a^b\frac{u-a}{u\log^2u}\,\dd u.
\]
Subtracting $(b-a)/\log b$ proves \eqref{eq:cell-identity}, including
nonnegativity. The function $g(u)=1/(u\log^2u)$ satisfies $g'(u)=-\frac{\log u+2}{u^2\log^3u}<0\;(u>1)$. It is therefore decreasing. Hence
\[
\frac1{b\log^2b}\int_a^b(u-a)\dd u
\leq d_j^{(\ell)}\leq
\frac1{a\log^2a}\int_a^b(u-a)\dd u.
\]
Since $\int_a^b(u-a)\dd u=(b-a)^2/2$, this proves \eqref{eq:cell-two-sided}.

From \eqref{eq:P-def} and \eqref{eq:h-def},
\[
P_\ell(j)=\frac{\ell-2}{2}j^2\left(1+\bigO_\ell(j^{-1})\right),
\qquad
h_j^{(\ell)}=(\ell-2)j\left(1+\bigO_\ell(j^{-1})\right).
\]
Also,
\[
\frac{P_\ell(j+1)}{P_\ell(j)}=1+\bigO_\ell(j^{-1}),
\qquad
\frac{\log P_\ell(j+1)}{\log P_\ell(j)}
 =1+\bigO_\ell\!\left(\frac1{j\log j}\right).
\]
Thus both bounds in \eqref{eq:cell-two-sided} equal $\frac{\ell-2}{\log^2P_\ell(j+1)} \left(1+\bigO_\ell(j^{-1})\right)$, which proves \eqref{eq:dj-asymptotic}. These expansions initially hold
for $j\geq J_\ell$ with one fixed threshold $J_\ell$. The finitely many
indices $2\leq j<J_\ell$ can be included by enlarging the implied constant,
because $1/(j\log^2(j+1))>0$ at each such index. Equivalently, their
contribution can be retained as a separate constant. No asymptotic
approximation at a small index is being assumed.
\end{proof}

\begin{lemma}[Terminal partial cell]\label[lemma]{lem:partial-cell}
For $x>P_\ell(2)$ and $n=n_x^{(\ell)}$, as $x\to\infty$,
\begin{equation}\label{eq:partial-bound}
0\leq d_{n-1}^{(\ell)}(x)
\leq d_{n-1}^{(\ell)}
=\bigO_\ell\!\left(\frac1{\log^2x}\right).
\end{equation}
\end{lemma}

\begin{proof}
For $P_\ell(n-1)\leq t\leq x\leq P_\ell(n)$, $\frac1{\log t}-\frac1{\log P_\ell(n)}\geq0$. Integrating over $[P_\ell(n-1),x]$ proves nonnegativity. Extending the interval of integration to $P_\ell(n)$ gives the upper bound by the complete-cell error. The final estimate follows from \eqref{eq:cell-two-sided}, \eqref{eq:n-asymptotic}, and $P_\ell(n-1)\asymp_\ell x$.
\end{proof}

\subsection{A slowly varying logarithmic sum}

\begin{lemma}[Summation lemma]\label[lemma]{lem:slow-sum}
As $N\to\infty$,
\begin{equation}\label{eq:slow-sum}
\sum_{j=2}^{N}\frac1{\log^2P_\ell(j+1)}
 =\frac{N}{\log^2P_\ell(N+1)}
  +\bigO_\ell\!\left(\frac{N}{\log^3N}\right).
\end{equation}
\end{lemma}

\begin{proof}
For real $t\geq2$, let $f(t)=\frac1{\log^2P_\ell(t+1)}$, where $P_\ell(t)$ is the polynomial in \eqref{eq:P-def}. The function $f$ is positive and decreasing, and
\[
f'(t)
 =-\frac{2(P_\ell(t+1))'}
          {P_\ell(t+1)\log^3P_\ell(t+1)}
 =\bigO_\ell\!\left(\frac1{t\log^3(t+1)}\right).
\]
Monotone integral comparison gives
\[
\sum_{j=2}^{N}f(j)=\int_2^N f(t)\dd t+\bigO_\ell(1).
\]
Integration by parts yields
\begin{align*}
\int_2^N f(t)\dd t
 &=Nf(N)-2f(2)-\int_2^N t f'(t)\dd t\\
 &=Nf(N)+\bigO_\ell\!\left(
      1+\int_2^N\frac{\dd t}{\log^3(t+1)}
    \right).
\end{align*}
To bound the remaining integral explicitly, for $N\geq16$ split it at
$\sqrt N$. The first part is $\bigO(\sqrt N)$, while
$\log(t+1)\geq\tfrac12\log N$ on the second part, giving
\[
\int_2^N\frac{\dd t}{\log^3(t+1)}
 \leq C\sqrt N+\frac{8N}{\log^3N}
 =\bigO\!\left(\frac{N}{\log^3N}\right).
\]
Here $\sqrt N=\bigO(N/\log^3N)$ since $\log^3N/\sqrt N\to0$.
The fixed endpoint terms are absorbed by the same remainder. This proves
\eqref{eq:slow-sum}.
\end{proof}

\subsection{Asymptotics of the accumulated error \texorpdfstring{$D_\ell(x)$}{D-ell(x)}}

\begin{theorem}[Accumulated discretization error]\label{thm:D-asymptotic}
For every fixed $\ell\geq3$,
\begin{equation}\label{eq:D-asymptotic}
D_\ell(x)
 =\frac{\sqrt{2(\ell-2)x}}{\log^2x}
  +\bigO_\ell\!\left(\frac{\sqrt{x}}{\log^3x}\right),
\qquad x\to\infty.
\end{equation}
\end{theorem}

\begin{proof}
Let $n=n_x^{(\ell)}$ and take $x$ large enough that $n\geq4$.
Combining \eqref{eq:D-exact}, \eqref{eq:dj-asymptotic}, and
\cref{lem:partial-cell} gives
\[
D_\ell(x)
 =(\ell-2)\sum_{j=2}^{n-2}\frac1{\log^2P_\ell(j+1)}
 +\bigO_\ell\!\left(1+\sum_{j=2}^{n-2}\frac1{j\log^2(j+1)}
                        +\frac1{\log^2x}\right).
\]
The remainder is $\bigO_\ell(1)$, because
\[
\sum_{j\geq2}\frac1{j\log^2(j+1)}<\infty,\qquad
\int_2^\infty\frac{\dd t}{t\log^2t}=\frac1{\log2}<\infty.
\]
The fixed initial contribution and the finitely many small indices are
included in this constant. Adding the single summand with $j=n-1$
changes the expression by only $\bigO_\ell(1/\log^2x)$, so
\begin{equation}\label{eq:D-sum-stage}
D_\ell(x)=(\ell-2)\sum_{j=2}^{n-1}\frac1{\log^2P_\ell(j+1)}
 +\bigO_\ell(1).
\end{equation}
This choice of upper index makes \cref{lem:slow-sum}, with $N=n-1$,
end at the same node $P_\ell(n)$ as the denominator $q_\ell(x)$:
\begin{align}
D_\ell(x)
 &=\frac{(\ell-2)(n-1)}{\log^2P_\ell(n)}
   +\bigO_\ell\!\left(1+\frac{n}{\log^3n}\right)\notag\\
 &=\frac{(\ell-2)n}{\log^2P_\ell(n)}
   +\bigO_\ell\!\left(\frac{n}{\log^3n}\right).
\label{eq:D-sum-stage2}
\end{align}
Indeed, $n/\log^3n\to\infty$, and replacing $n-1$ by $n$ adds only
$(\ell-2)/\log^2P_\ell(n)=\bigO_\ell(1/\log^2n)$.

By \eqref{eq:n-asymptotic} and \eqref{eq:log-terminal},
\[
(\ell-2)n=\sqrt{2(\ell-2)x}+\bigO_\ell(1),\qquad
\log P_\ell(n)=\log x+\bigO_\ell(x^{-1/2}).
\]
The mean-value theorem for $u\mapsto u^{-2}$ then gives
\[
\frac1{(\log P_\ell(n))^2}
 =\frac1{\log^2x}
  +\bigO_\ell\!\left(\frac1{\sqrt{x}\log^3x}\right).
\]
After multiplication by $(\ell-2)n=\bigO_\ell(\sqrt{x})$, the last
replacement costs $\bigO_\ell(1/\log^3x)$; the numerator replacement
costs $\bigO_\ell(1/\log^2x)$. Both are absorbed by the stated error.
Finally, $n=\bigO_\ell(\sqrt{x})$ and
$\log n=\tfrac12\log x+\bigO_\ell(1)$ turn the remainder in
\eqref{eq:D-sum-stage2} into $\bigO_\ell(\sqrt{x}/\log^3x)$.
This proves \eqref{eq:D-asymptotic}, uniformly with respect to the
position of $x$ within its cell. In particular, $D_\ell(x)>0$ for all
sufficiently large $x$; no global sign assertion on the initial interval
is needed.
\end{proof}

\subsection{The unconditional transfer theorem}

\begin{theorem}[Unconditional relation between $Z_\ell$ and $E$]\label{thm:transfer}
For every fixed $\ell\geq3$,
\begin{equation}\label{eq:transfer-precise}
Z_\ell(x)
 =\left(\frac1{\sqrt2}+\bigO_\ell(x^{-1/2})\right)E(x)
  -\frac{\sqrt{\ell-2}}{\log x}
  +\bigO_\ell\!\left(\frac1{\log^2x}\right).
\end{equation}
Using the classical zero-free-region remainder in the prime number theorem, this simplifies unconditionally to
\begin{equation}\label{eq:transfer-simple}
\boxed{
Z_\ell(x)
 =\frac1{\sqrt2}E(x)
  -\frac{\sqrt{\ell-2}}{\log x}
  +\bigO_\ell\!\left(\frac1{\log^2x}\right)
}
\qquad (x\to\infty).
\end{equation}
\end{theorem}

\begin{proof}
By definitions \eqref{eq:D-def}, \eqref{eq:Z-def}, and \eqref{eq:E-def},
\begin{equation}\label{eq:Z-split}
Z_\ell(x)
 =\frac{\Li(x)-\pi(x)}{q_\ell(x)}
  -\frac{D_\ell(x)}{q_\ell(x)}.
\end{equation}
The denominator asymptotic \eqref{eq:q-asymptotic} gives
\begin{equation}\label{eq:prime-part}
\frac{\Li(x)-\pi(x)}{q_\ell(x)}
 =\left(\frac1{\sqrt2}+\bigO_\ell(x^{-1/2})\right)E(x).
\end{equation}
Next, \cref{thm:D-asymptotic,lem:q-asymptotic} imply
\begin{align}
\frac{D_\ell(x)}{q_\ell(x)}
 &=\frac{
     \dfrac{\sqrt{2(\ell-2)x}}{\log^2x}
       +\bigO_\ell(\sqrt{x}/\log^3x)
    }{
     \dfrac{\sqrt{2x}}{\log x}
       \left(1+\bigO_\ell(x^{-1/2})\right)
    }\notag\\
 &=\frac{\sqrt{\ell-2}}{\log x}
   +\bigO_\ell\!\left(\frac1{\log^2x}\right).
\label{eq:D-over-q}
\end{align}
Substitution into \eqref{eq:Z-split} proves \eqref{eq:transfer-precise}.

Finally, the classical de la Vall\'ee Poussin zero-free region implies, for some $a_0>0$,
\begin{equation}\label{eq:PNT-remainder}
\Li(x)-\pi(x)=\bigO\!\left(xe^{-a_0\sqrt{\log x}}\right);
\end{equation}
an explicit estimate for $\operatorname{Ei}(\log x)-\pi(x)$
implying this form is given in \cite[Theorem~2]{Trudgian}. Subtracting
the fixed constant $C_*$ preserves the stated remainder. The argument only needs some fixed $a_0>0$;
it does not use a numerically optimized zero-free region. Hence
\[
\frac{|E(x)|}{\sqrt{x}}
 =\bigO\!\left(\log x\,e^{-a_0\sqrt{\log x}}\right)
 =\bigO\!\left(\frac1{\log^2x}\right).
\]
The coefficient error in \eqref{eq:transfer-precise} can therefore be absorbed into the final remainder. No form of RH has been used.
\end{proof}

\subsection{Logarithmic means}

\begin{definition}[Logarithmic mean]\label[definition]{def:log-mean}
For a locally integrable function $f$ on $[2,\infty)$, define
\begin{equation}\label{eq:log-mean-def}
\Mcal_X[f]
 :=\frac1{\log(X/2)}\int_2^X f(x)\frac{\dd x}{x},
\qquad X>2.
\end{equation}
\end{definition}

\begin{theorem}[Unconditional transfer of logarithmic means]\label{thm:mean-transfer}
For every fixed $\ell\geq3$,
\begin{equation}\label{eq:mean-transfer-exact}
\Mcal_X[Z_\ell]
 =\frac1{\sqrt2}\Mcal_X[E]
  -\sqrt{\ell-2}\,
   \frac{\log\log X-\log\log2}{\log(X/2)}
  +\bigO_\ell\!\left(\frac1{\log X}\right).
\end{equation}
In particular,
\begin{equation}\label{eq:mean-transfer-short}
\Mcal_X[Z_\ell]
 -\frac1{\sqrt2}\Mcal_X[E]
 =\bigO_\ell\!\left(\frac{\log\log X}{\log X}\right).
\end{equation}
Consequently, one of the two logarithmic means has a finite limit if and only if the other does, and in that case
\begin{equation}\label{eq:mean-limit-scaling}
\lim_{X\to\infty}\Mcal_X[Z_\ell]
 =\frac1{\sqrt2}\lim_{X\to\infty}\Mcal_X[E].
\end{equation}
\end{theorem}

\begin{proof}
Integrate \eqref{eq:transfer-simple} against $\dd x/x$ from a fixed sufficiently large point to $X$. The omitted compact initial interval contributes $\bigO_\ell(1/\log X)$ after normalization. Moreover, $\int_2^X\frac{\dd x}{x\log x} =\log\log X-\log\log2$, and $\int_2^X\frac{\dd x}{x\log^2x} =\frac1{\log2}-\frac1{\log X}=\bigO(1)$. Division by $\log(X/2)$ proves \eqref{eq:mean-transfer-exact}; the remaining statements follow immediately.
\end{proof}

\section{Logarithmic means and the Riemann hypothesis}\label{sec:RH}

\subsection{Sources for the classical forward implication}

Wintner's papers of 1935 and 1941 \cite{Wintner,Wintner1941} initiated
the almost-periodic study of the prime-number-theorem remainder under RH, initially for
Chebyshev's function $\psi$. Rubinstein and Sarnak \cite{RubinsteinSarnak}
developed the corresponding distributional approach to prime-counting
errors.

A modern reference for the prime-counting normalization is Devin's paper
\cite{Devin}, published in \emph{Mathematical Proceedings of the Cambridge
Philosophical Society} \textbf{169}(1) (2020), 103--140.

\begin{theorem}[Classical logarithmic mean under RH]\label{thm:classical-forward}
Assume the Riemann hypothesis. With $E$ as in \eqref{eq:E-def},
$\Li(x)=\int_2^x\dd t/\log t$, and the logarithmic mean
\eqref{eq:log-mean-def},
\begin{equation}\label{eq:E-mean-under-RH}
\lim_{X\to\infty}\Mcal_X[E]
=\lim_{X\to\infty}\frac1{\log(X/2)}
 \int_2^X\frac{\log x}{\sqrt x}
          \bigl(\Li(x)-\pi(x)\bigr)\frac{\dd x}{x}
=1.
\end{equation}
No assumption of simplicity or linear independence of the ordinates of
nontrivial zeros is required.
\end{theorem}

We quote this classical result, rather than prove it here. Its sources
are cited above, and the normalization is specified below. The converse is proved
separately from a bound on a signed primitive and does not use RH or
\cref{thm:classical-forward}. Johnston's earlier oscillation theorem
\cite[Lemma~2.8]{Johnston2023} also yields the converse; the complete
reduction is given in \cref{subsubsec:Johnston-reduction}. Our direct
argument replaces that oscillation step by a local pole contradiction
and does not continue logarithmic branches; it does not prove Johnston's
stronger oscillation estimates.

Passage to the mean uses mean-square approximation, not the existence of
a limiting distribution alone \cite[Section~2]{AkbaryNgShahabi}.

For fixed $C\in\mathbb R$, replacing $\Li$ by $\Li+C$ replaces $E(x)$
by $E^{[C]}(x):=E(x)+C(\log x)/\sqrt x$. Since
$\int_2^\infty(\log t)t^{-3/2}\dd t<\infty$,
\[
\Mcal_X[E^{[C]}]-\Mcal_X[E]
=\frac{C}{\log(X/2)}\int_2^X\frac{\log t}{t^{3/2}}\dd t
=\bigO_C\!\left(\frac1{\log X}\right)\longrightarrow0.
\]
Thus the limiting mean is unchanged, although finite-range values and signs
may change. Changing a fixed lower averaging endpoint only changes the
numerator by a constant and gives an asymptotically equivalent denominator.
The following consequence for the estimator uses only the quoted theorem
and the unconditional transfer already proved in Section~\ref{sec:coefficient}.

\begin{corollary}[Mean of the normalized S--S error]\label[corollary]{cor:one-over-sqrt2}
Assume the Riemann hypothesis. For every fixed $\ell\geq3$,
\begin{equation}\label{eq:Z-mean-under-RH}
\boxed{
\lim_{X\to\infty}\Mcal_X[Z_\ell]=\frac1{\sqrt2}
}.
\end{equation}
\end{corollary}

\begin{proof}
Combine \eqref{eq:E-mean-under-RH} with \cref{thm:mean-transfer}.
\end{proof}

\subsection{Direct converse: Mellin continuation}
\label{subsubsec:direct-Mellin}

For $x\geq2$ and $X\geq2$, set
\begin{equation}\label{eq:converse-primitive}
\begin{split}
\Delta(x)&:=\Li(x)-\pi(x),\\
\mathcal A_E(X)&:=\int_2^X\Delta(x)\frac{\log x}{x^{3/2}}\dd x
                =\int_2^X E(x)\frac{\dd x}{x}.
\end{split}
\end{equation}
The function $\Delta$ is locally integrable, so $\mathcal A_E$ is locally
absolutely continuous, $\mathcal A_E(2)=0$, and
$\mathcal A_E'(x)=\Delta(x)(\log x)x^{-3/2}$ almost everywhere.
The values of $\pi$ at its jump points do not affect these integrals.

The direct converse will use the following condition:
\begin{equation}\label{eq:converse-subpower}
\mathcal A_E(X)=\bigO_\varepsilon(X^\varepsilon)
\quad\text{for every }\varepsilon>0.
\end{equation}
We call this \emph{subpower growth}. It bounds the modulus of the signed
primitive, not the integral of $|E|$. The constant may depend on
$\varepsilon$; the quantifier ``for every'' is essential. Local
boundedness allows the estimate to be extended to all $X\geq2$ by
increasing its constant. In particular,
\begin{equation}\label{eq:bounded-mean-subpower}
\Mcal_X[E]=\bigO(1)
\ \Longrightarrow\ 
\mathcal A_E(X)=\bigO(\log X)
\ \Longrightarrow\ \eqref{eq:converse-subpower}.
\end{equation}
For the last implication, $\log X\leq X^\varepsilon/(e\varepsilon)$
for $X\geq1$, as follows by maximizing $u e^{-\varepsilon u}$.

Write $\mathcal D=\{z\in\C:\Re z>1/2\}$, $\mathcal D_1=\{z\in\C:\Re z>1\}$, and use $x^{-z}=\exp(-z\log x)$ for positive $x$. The elementary bounds
$\pi(x)\leq x$ and $0\leq\Li(x)\leq(x-2)/\log2$ give
$\Delta(x)=\bigO(x)$. Hence
\begin{equation}\label{eq:converse-J-original}
J(z):=\int_2^\infty\Delta(x)x^{-z-1}\dd x
\end{equation}
is absolutely convergent and holomorphic on $\mathcal D_1$.

\begin{lemma}[Continuation from the signed primitive]\label[lemma]{lem:converse-Mellin}
Assume \eqref{eq:converse-subpower}. Then $J$ has the holomorphic
continuation
\begin{equation}\label{eq:converse-Mellin}
J(z)=\int_2^\infty\mathcal A_E(x)x^{-z-1/2}
 \left(\frac{z-1/2}{\log x}+\frac1{\log^2x}\right)\dd x,
\qquad z\in\mathcal D.
\end{equation}
The new integral converges absolutely and locally uniformly on
$\mathcal D$. Its derivative is
\begin{equation}\label{eq:converse-J-derivative}
J'(z)=-(z-\tfrac12)\int_2^\infty
               \mathcal A_E(x)x^{-z-1/2}\dd x,
\qquad z\in\mathcal D.
\end{equation}
\end{lemma}

\begin{proof}
Introduce
\[
k_z(x)=\frac{x^{1/2-z}}{\log x},\qquad
k_z'(x)=-x^{-z-1/2}
 \left(\frac{z-1/2}{\log x}+\frac1{\log^2x}\right).
\]
Since $k_z(x)\mathcal A_E'(x)=\Delta(x)x^{-z-1}$ almost everywhere,
integration by parts on $[2,T]$ gives the exact finite-interval identity
\begin{equation}\label{eq:converse-finite-Mellin}
\begin{split}
\int_2^T\Delta(x)x^{-z-1}\dd x
={}&\frac{\mathcal A_E(T)T^{1/2-z}}{\log T}\\
 &+\int_2^T\mathcal A_E(x)x^{-z-1/2}
   \left(\frac{z-1/2}{\log x}+\frac1{\log^2x}\right)\dd x.
\end{split}
\end{equation}
There is no lower boundary term because $\mathcal A_E(2)=0$.

Fix a compact set $K\subset\mathcal D$ and let $\delta=\min_{z\in K}(\Re z-\tfrac12)>0, \; M=\max_{z\in K}|z-\tfrac12|$. Apply \eqref{eq:converse-subpower} with $\varepsilon=\delta/2$.
For some constant $C_K$, $|\mathcal A_E(x)|\leq C_Kx^{\delta/2}$
for all $x\geq2$. Uniformly for $z\in K$, the modulus of the integrand
in \eqref{eq:converse-Mellin} is at most
\begin{equation}\label{eq:converse-Mellin-majorant}
C_K\left(\frac M{\log2}+\frac1{(\log2)^2}\right)x^{-1-\delta/2}.
\end{equation}
This is integrable. The boundary term in
\eqref{eq:converse-finite-Mellin} is at most
$C_K T^{-\delta/2}/\log T$ and tends to zero uniformly on $K$.
Taking $T\to\infty$ proves \eqref{eq:converse-Mellin} on
$\mathcal D_1$ and gives locally uniform convergence throughout
$\mathcal D$. The finite integrals on its right are entire in $z$;
their locally uniform limit is holomorphic. The argument also proves
ordinary improper convergence of \eqref{eq:converse-J-original}
on $\mathcal D$, but does not assert its absolute convergence there.

For the derivative, there is an exact cancellation:
\[
\frac{\partial}{\partial z}\left[
 x^{-z-1/2}\left(\frac{z-1/2}{\log x}+\frac1{\log^2x}\right)\right]
 =-(z-\tfrac12)x^{-z-1/2}.
\]
After multiplication by $\mathcal A_E(x)$, the derivative is dominated
on $K$ by $C_KM x^{-1-\delta/2}$. Differentiation under the integral
is therefore justified and gives \eqref{eq:converse-J-derivative}.
\end{proof}

\subsection{Mellin identities and the pole obstruction}
\label{subsubsec:prime-zeta-identities}

On $\mathcal D_1$ define
\begin{equation}\label{eq:converse-PI-def}
\mathcal P(z):=\sum_p p^{-z},\qquad
I(z):=\int_2^\infty\frac{x^{-z}}{\log x}\dd x.
\end{equation}
The first is the prime zeta function \cite{Froberg}; the calligraphic
symbol distinguishes it from the polygonal numbers $P_\ell(n)$.
Both expressions are holomorphic on $\mathcal D_1$. Indeed, for
$\sigma_0>1$ and every nonnegative integer $j$,
\[
\sum_{n\geq2}\frac{(\log n)^j}{n^{\sigma_0}}<\infty,
\qquad
\int_2^\infty x^{-\sigma_0}(\log x)^j\dd x<\infty.
\]
These majorants also justify the differentiations on compact subsets of
$\mathcal D_1$ used below.

\begin{lemma}[Mellin identities on the initial half-plane]\label[lemma]{lem:converse-PI}
For $z\in\mathcal D_1$,
\begin{equation}\label{eq:converse-Mellin-identities}
zJ(z)=I(z)-\mathcal P(z),\qquad
I'(z)=-\frac{2^{1-z}}{z-1}.
\end{equation}
\end{lemma}

\begin{proof}
Put $\sigma=\Re z>1$. Since
\[
\sum_p\int_p^\infty|x^{-z-1}|\dd x
  =\frac1\sigma\sum_p p^{-\sigma}<\infty,
\]
we may insert $\pi(x)=\sum_p\mathbf1_{[p,\infty)}(x)$ and apply
Fubini's theorem to obtain
\[
\int_2^\infty\pi(x)x^{-z-1}\dd x
 =\sum_p\int_p^\infty x^{-z-1}\dd x=\frac{\mathcal P(z)}z.
\]
For the logarithmic integral, the corresponding absolute double integral
is finite:
\[
\int_2^\infty\frac1{\log t}
       \int_t^\infty x^{-\sigma-1}\dd x\dd t
 =\frac1\sigma\int_2^\infty\frac{t^{-\sigma}}{\log t}\dd t<\infty.
\]
Consequently,
\[
\int_2^\infty\Li(x)x^{-z-1}\dd x
 =\int_2^\infty\frac1{\log t}
       \int_t^\infty x^{-z-1}\dd x\dd t=\frac{I(z)}z.
\]
Subtracting yields the first identity. The convention $\Li(2)=0$
accounts for the absence of an additive endpoint term. Finally,
differentiation under the integral gives
\[
I'(z)=-\int_2^\infty x^{-z}\dd x=-\frac{2^{1-z}}{z-1}.
\]
\end{proof}

\begin{lemma}[Holomorphic proper-prime-power remainder]\label[lemma]{lem:converse-prime-powers}
The series
\begin{equation}\label{eq:converse-H-def}
\mathcal H(z):=\sum_p\sum_{m\geq2}(\log p)p^{-mz}
\end{equation}
defines a holomorphic function on $\mathcal D$. On $\mathcal D_1$,
\begin{equation}\label{eq:converse-Euler}
\mathcal P'(z)=\frac{\zeta'(z)}{\zeta(z)}+\mathcal H(z).
\end{equation}
\end{lemma}

\begin{proof}
For a compact set $K\subset\mathcal D$, put
$\sigma_0=\min_{z\in K}\Re z>1/2$. Summing the geometric series gives
\begin{align}
\sum_p\sum_{m\geq2}\sup_{z\in K}|(\log p)p^{-mz}|
 &\leq\sum_p\frac{(\log p)p^{-2\sigma_0}}{1-p^{-\sigma_0}}\notag\\
 &\leq\frac1{1-2^{-\sigma_0}}
       \sum_{n\geq2}\frac{\log n}{n^{2\sigma_0}}<\infty.
\label{eq:converse-H-majorant}
\end{align}
The last series converges because $2\sigma_0>1$. Thus the entire
summands converge locally normally, meaning that the sum of their
suprema is finite on each compact subset. Their sum is holomorphic.

On $\mathcal D_1$, the logarithmic derivative of the Euler product
\cite{Titchmarsh} is
\[
-\frac{\zeta'(z)}{\zeta(z)}
 =\sum_{n\geq2}\frac{\Lambda(n)}{n^z}
 =\sum_p\sum_{m\geq1}(\log p)p^{-mz},
\]
where $\Lambda(p^m)=\log p$ and $\Lambda(n)=0$ otherwise.
The series is absolutely convergent. Separating the term $m=1$ and
using $\mathcal P'(z)=-\sum_p(\log p)p^{-z}$ proves
\eqref{eq:converse-Euler}. No logarithmic branch of $\zeta$ is
continued in this step.
\end{proof}

\paragraph{Converse criterion.}
\label{subsubsec:converse-proof}

\begin{proposition}[Subpower signed-primitive criterion]\label[proposition]{prop:subpower-converse}
If \eqref{eq:converse-subpower} holds, then RH holds. In particular,
$\Mcal_X[E]=\bigO(1)$, or existence of any finite limit of
$\Mcal_X[E]$, implies RH.
\end{proposition}

\begin{proof}
Assume \eqref{eq:converse-subpower}. By \cref{lem:converse-Mellin},
$J$ and $J'$ are holomorphic on $\mathcal D$. Differentiating the first
identity in \eqref{eq:converse-Mellin-identities} on $\mathcal D_1$
yields $\mathcal P'=I'-J-zJ'$. Hence
\begin{equation}\label{eq:converse-Phi}
\Phi(z):=-\frac{2^{1-z}}{z-1}-J(z)-zJ'(z),\qquad z\in\mathcal D,
\end{equation}
is a meromorphic continuation of $\mathcal P'$ and is holomorphic on
$\mathcal D\setminus\{1\}$. Its only possible pole is at $1$; in fact, $-\frac{2^{1-z}}{z-1}=-\frac1{z-1}+\log2+\bigO(z-1)$, so the pole there is simple with residue $-1$.

The independently defined function $Q(z):=\frac{\zeta'(z)}{\zeta(z)}+\mathcal H(z)$ is meromorphic on $\mathcal D$. Meromorphicity of $\zeta'/\zeta$
follows from the unconditional continuation of $\zeta$, and
\cref{lem:converse-prime-powers} supplies the holomorphic remainder.
On $\mathcal D_1$, both $Q$ and $\Phi$ equal $\mathcal P'$.
The identity theorem for meromorphic functions on the connected domain
$\mathcal D$ therefore gives
\begin{equation}\label{eq:converse-pole-identity}
\Phi(z)=\frac{\zeta'(z)}{\zeta(z)}+\mathcal H(z)
\quad\text{as meromorphic functions on }\mathcal D.
\end{equation}
This statement compares continuations of the derivative, not branches
of $\mathcal P$ or of $\log\zeta$ across $z=1$.

Suppose $\rho\in\mathcal D$ is a zero of $\zeta$ with multiplicity
$m_\rho\geq1$. Since $\rho\ne1$, $\Phi$ is holomorphic at $\rho$.
In a sufficiently small disc we may factor $\zeta(z)=(z-\rho)^{m_\rho}g(z)$, $g(\rho)\ne0$, where $g$ is holomorphic and nonvanishing. It follows that
\begin{equation}\label{eq:converse-local-pole}
Q(z)=\frac{m_\rho}{z-\rho}+\frac{g'(z)}{g(z)}+\mathcal H(z).
\end{equation}
The last two terms are holomorphic in this disc. They cannot cancel
the principal part. Multiplying \eqref{eq:converse-pole-identity} by
$z-\rho$ and letting $z\to\rho$ gives $0=m_\rho$, a contradiction.
Thus there is no zero with real part greater than $1/2$.

To conclude, the completed zeta function $\xi(z)=\tfrac12z(z-1)\pi^{-z/2}\Gamma(z/2)\zeta(z)$ is entire, has precisely the nontrivial zeros of $\zeta$, and satisfies
$\xi(z)=\xi(1-z)$ \cite{Titchmarsh}. A nontrivial zero to the left of
the critical line would therefore give one to its right. Every
nontrivial zero has real part $1/2$, which is RH.

Finally, boundedness of $\Mcal_X[E]$ implies
\eqref{eq:converse-subpower} by \eqref{eq:bounded-mean-subpower}, and
a finite limit implies eventual boundedness. The value of a putative
limit is not used to prove the converse.
\end{proof}

The contradiction is local at any hypothetical zero, including a multiple
zero. The Euler product is used only for $\Re z>1$. Neither the
classical forward theorem \cref{thm:classical-forward} nor Johnston's oscillation theorem
enters this converse.

\subsection{The RH equivalence and the centering coefficient}\label{subsec:RH-criterion}

\begin{theorem}[RH equivalence for the normalized S--S error]\label{thm:RH-equivalence}
For every fixed integer $\ell\geq3$,
\begin{equation}\label{eq:RH-equivalence}
\boxed{\quad
\mathrm{RH}\quad\Longleftrightarrow\quad
\lim_{X\to\infty}\Mcal_X[Z_\ell]=\frac1{\sqrt2}.
\quad}
\end{equation}
Each of the following conditions is also equivalent to RH: the existence
of an unspecified finite real limit of $\Mcal_X[Z_\ell]$, and the bound
$\Mcal_X[Z_\ell]=\bigO_\ell(1)$ as $X\to\infty$. Any finite limit must
therefore equal $1/\sqrt2$. Equivalently,
\begin{equation}\label{eq:centered-RH-equivalence}
\mathrm{RH}\quad\Longleftrightarrow\quad
\lim_{X\to\infty}\Mcal_X\!\left[
\frac{S_{\ell,1/\sqrt2}-\pi}{q_\ell}\right]=0.
\end{equation}
The condition for one fixed $\ell$ suffices; under RH it holds for every
fixed $\ell\geq3$.
\end{theorem}

\begin{proof}
Under RH, \cref{cor:one-over-sqrt2} gives the limit $1/\sqrt2$.
Convergence to this value implies a finite limit, and any finite limit
implies eventual boundedness.

Conversely, fix $\ell\geq3$ and suppose
$\Mcal_X[Z_\ell]=\bigO_\ell(1)$. The unconditional relation
\eqref{eq:mean-transfer-short} gives
\[
\Mcal_X[E]=\sqrt2\,\Mcal_X[Z_\ell]
 +\bigO_\ell\!\left(\frac{\log\log X}{\log X}\right)
 =\bigO_\ell(1).
\]
Thus \cref{prop:subpower-converse} proves RH. Applying the already
established forward implication then forces the mean of $Z_\ell$ to
converge to $1/\sqrt2$. This proves all the asserted equivalences and
identifies any unspecified finite limit.

The centered version follows from the exact identity
\[
\Mcal_X\!\left[\frac{S_{\ell,1/\sqrt2}-\pi}{q_\ell}\right]
 =\Mcal_X[Z_\ell]-\frac1{\sqrt2},
\]
because $\Mcal_X[1]=1$. Only one fixed value of $\ell$ was needed
for the converse; the forward result applies to every fixed
$\ell\geq3$.
\end{proof}

\begin{theorem}[Asymptotic optimality of $1/\sqrt2$]\label{thm:coefficient-optimality}
Assume the Riemann hypothesis. For every fixed $\ell\geq3$ and every fixed $c\in\mathbb R$,
\begin{equation}\label{eq:coefficient-bias}
\lim_{X\to\infty}
\Mcal_X\!\left[
\frac{S_{\ell,c}(x)-\pi(x)}{q_\ell(x)}
\right]
 =\frac1{\sqrt2}-c.
\end{equation}
Consequently,
\begin{equation}\label{eq:c-theoretical}
c_{\mathrm{th}}=\frac1{\sqrt2}
\end{equation}
is the unique fixed coefficient that gives zero asymptotic logarithmic-mean bias. Equivalently, it uniquely minimizes the absolute limiting bias $\left|\frac1{\sqrt2}-c\right|$. In this precise sense, $S_{\ell,1/\sqrt2}$ is the asymptotically optimal member of the S--S correction family.
\end{theorem}

\begin{proof}
By \eqref{eq:S-family-def}, $\frac{S_{\ell,c}(x)-\pi(x)}{q_\ell(x)}=Z_\ell(x)-c$. The logarithmic mean is linear and $\Mcal_X[1]=1$ exactly, so $\Mcal_X\!\left[\frac{S_{\ell,c}-\pi}{q_\ell}\right] =\Mcal_X[Z_\ell]-c$. Apply \cref{cor:one-over-sqrt2} and pass to the limit. The limiting bias
is zero precisely when $c=1/\sqrt2$, and its absolute value is strictly
positive for every other fixed coefficient. This proves uniqueness.
The statement concerns a \emph{signed logarithmic-mean bias}; it does not
claim pointwise, minimax, or finite-sample absolute-error optimality.
\end{proof}

\section{Comparison with \texorpdfstring{$R(x)$}{R(x)} and numerical results}\label{sec:numerics}

\subsection{The Riemann approximation and main-term equivalence}

To avoid any ambiguity about truncation or summation conventions, define the
Riemann prime-counting approximation by its convergent Gram series
\begin{equation}\label{eq:R-def}
R(x)=1+\sum_{k=1}^{\infty}
 \frac{(\log x)^k}{k\,k!\,\zeta(k+1)},\qquad x>1;
\end{equation}
see \cite{Mpmath}. Here $\zeta$ is the Riemann zeta function. Since
$\zeta(k+1)>1$, absolute convergence follows by comparison with
$\sum_{k\geq1}(\log x)^k/k!$. In particular, $R(x)$ is real-valued,
not integer-valued.

\begin{proposition}[Main-term equivalence with $R(x)$]\label[proposition]{prop:R-equivalence}
For every fixed $\ell\geq3$ and every fixed $c\in\mathbb R$,
\begin{equation}\label{eq:R-equivalence}
S_{\ell,c}(x)\sim\widehat S_\ell(x)\sim\Li(x)\sim\pi(x)\sim R(x),
\qquad x\to\infty.
\end{equation}
In particular,
\begin{equation}\label{eq:S-over-R}
\lim_{x\to\infty}\frac{S_{\ell,c}(x)}{R(x)}=1.
\end{equation}
\end{proposition}

\begin{proof}
By \cref{thm:D-asymptotic},
\[
\Li(x)-\widehat S_\ell(x)
 =\bigO_\ell\!\left(\frac{\sqrt{x}}{\log^2x}\right)
 =o\!\left(\frac{x}{\log x}\right).
\]
Hence $\widehat S_\ell(x)\sim\Li(x)$. By \cref{lem:q-asymptotic},
\[
c\,q_\ell(x)
 =\bigO_{\ell,c}\!\left(\frac{\sqrt{x}}{\log x}\right)
 =o\!\left(\frac{x}{\log x}\right),
\]
so $S_{\ell,c}(x)\sim\widehat S_\ell(x)$. For completeness, the main term of $R$ can be obtained directly from
\eqref{eq:R-def}, rather than from a formal interchange of an infinite
M\"obius sum. Put $y=\log x>0$. The exponential-integral expansion
\cite[Eq.~6.6.1]{DLMF} is
\[
\Li(x)=\gamma+\log y-C_*+\sum_{k=1}^{\infty}\frac{y^k}{k\,k!},
\]
where $\gamma$ is Euler's constant. For $k\geq1$,
\[
0\leq1-\frac1{\zeta(k+1)}\leq\zeta(k+1)-1
 \leq 2^{-k-1}+\int_2^\infty t^{-k-1}\dd t
 \leq\frac32\,2^{-k}.
\]
Subtract the absolutely convergent series and use $1/k\leq1$ to get
\[
\begin{aligned}
|R(x)-\Li(x)|
 &\leq |1-\gamma-\log y+C_*|+\frac32\sum_{k=1}^{\infty}\frac{(y/2)^k}{k\,k!}\\
 &\leq |1-\gamma-\log y+C_*|+\frac32(e^{y/2}-1)
 =\bigO(\sqrt x).
\end{aligned}
\]
Because $\sqrt x=o(x/\log x)$, this proves
$R(x)\sim\Li(x)\sim x/\log x$. The prime number theorem gives
$\pi(x)\sim\Li(x)$. Dividing each estimator by $x/\log x$ now proves
\eqref{eq:R-equivalence}; taking the quotient of two quantities tending
to $1$ proves \eqref{eq:S-over-R}.
\end{proof}

\subsection{Implementation and numerical checks}

The decimal coefficient $c_T$ and the exact theoretical coefficient answer related but not identical questions. The former specifies the computation reproduced in the tables; the latter is derived from the asymptotic relation between the natural scale of the prime-counting error and the terminal scale of the S--S estimator; see \cref{sec:coefficient,sec:RH}. Numerically, $\frac1{\sqrt2}-c_T\approx6.7812\times10^{-6}$, but a small coefficient difference need not give a small \emph{absolute}
change at a large argument. Exactly,
\begin{equation}\label{eq:coefficient-change}
S_{\ell,c_T}(x)-S_{\ell,1/\sqrt2}(x)
 =\left(\frac1{\sqrt2}-c_T\right)q_\ell(x).
\end{equation}
For $\ell=3$ and $x=10^{19}$ this difference is approximately $693.19$.
Thus the table for $c_T$ must not be relabelled as a computation with the
exact theoretical coefficient.

The initial experiments were carried out in Google Colab using Python~3
on a CPU. The implementation used \texttt{decimal} with 30 significant
digits, together with \texttt{math}, NumPy, Pandas, and Matplotlib.
The reference values of $\pi(10^k)$ are treated as exact input data;
for $10^{10}\leq x\leq10^{19}$ they also agree with the table maintained
by the \texttt{primecount} project \cite{Primecount}. We did not independently
recount the primes up to the largest arguments.

The tables use exactly $c_T=0.7071$ and
$\widetilde S_T(x)=S_{3,c_T}(x)$. The signed relative error is
\begin{equation}\label{eq:dT-def}
d_T(x)=\frac{\widetilde S_T(x)-\pi(x)}{\pi(x)}.
\end{equation}
For the baseline numerical validation we used \texttt{mpmath}~1.3.0 \cite{Mpmath}.
The values of $R$ were evaluated with 100-digit arithmetic and compared
with an explicit 500-term Gram sum. If $y=\log x$, the omitted tail satisfies
\[
\sum_{k>500}\frac{y^k}{k!}
 \leq\frac{y^{501}/501!}{1-y/502}<10^{-100},
 \qquad 10\leq x\leq10^{19}.
\]
This bounds the truncation error; it is not an interval certificate for
all floating-point operations.

In that validation, the S--S sums were checked by two Euler--Maclaurin evaluations: one with
60-digit arithmetic and splitting index 128, and another with 90 digits
and splitting index 512. Their absolute difference was below $10^{-23}$
at each of the 19 arguments. Direct 60-digit summation through $10^{10}$
provided a further check at the first ten points, with the same tolerance.
The resulting signed errors agree with the supplied numerical calculations
to two decimal places, and the relative errors agree to ten decimal places.
These checks support the reported precision but are not rigorous interval
enclosures. Euler--Maclaurin is used here to check large sums, not to define
the streaming algorithm or its complexity.

The supplementary audit for this revision rechecks the finite-sum identities,
integer boundary indices, and the tabulated values separately from the
historical Colab run. Its exact test range and tolerances are recorded in
the audit report; it is not a reconstruction of the original runtime.

All subtractions are performed before rounding. Two examples illustrate
why this matters: $R(1000)-168=0.359446281167\ldots$, $\widetilde S_T(10)-4=0.003629747905\ldots$. Neither value is zero, although the second is displayed as $0.00$ to two
decimal places. The supplementary files contain the higher-precision values,
verification scripts, and the original input strings. No wall-clock
comparison between algorithms is claimed.

\subsection{Accuracy and the effect of the correction}

Tables~\ref{tab:R-values} and~\ref{tab:estimator-values} give the
Riemann and corrected S--S approximations. Table~\ref{tab:correction-effect}
also includes the uncorrected error, making the size and direction of the
terminal correction visible. The repeated node information is listed
only in the first table. Relative errors are calculated from the unrounded
values and displayed in scientific notation.

\begingroup
\scriptsize
\setlength{\tabcolsep}{2.5pt}
\setlength{\LTpre}{6pt}
\setlength{\LTpost}{6pt}
\setlength{\LTcapwidth}{\textwidth}
\begin{longtable}{rrrrrr}
\caption{Partition indices, reference prime counts, and the real-valued
Riemann approximation. Values and signed errors are rounded to two decimal
places only after high-precision evaluation and subtraction.}
\label{tab:R-values}\\
\toprule
$x$ & $n_x$ & $T_{n_x}$ & $\pi(x)$ & $R(x)$ & $R(x)-\pi(x)$\\
\midrule
\endfirsthead
\toprule
$x$ & $n_x$ & $T_{n_x}$ & $\pi(x)$ & $R(x)$ & $R(x)-\pi(x)$\\
\midrule
\endhead
\bottomrule
\endfoot
\endlastfoot
$10^{1}$ & 4 & 10 & 4 & 4.56 & 0.56\\
$10^{2}$ & 14 & 105 & 25 & 25.66 & 0.66\\
$10^{3}$ & 45 & 1035 & 168 & 168.36 & 0.36\\
$10^{4}$ & 141 & 10011 & 1229 & 1226.93 & -2.07\\
$10^{5}$ & 447 & 100128 & 9592 & 9587.43 & -4.57\\
$10^{6}$ & 1414 & 1000405 & 78498 & 78527.40 & 29.40\\
$10^{7}$ & 4472 & 10001628 & 664579 & 664667.45 & 88.45\\
$10^{8}$ & 14142 & 100005153 & 5761455 & 5761551.87 & 96.87\\
$10^{9}$ & 44721 & 1000006281 & 50847534 & 50847455.43 & -78.57\\
$10^{10}$ & 141421 & 10000020331 & 455052511 & 455050683.31 & -1827.69\\
$10^{11}$ & 447214 & 100000404505 & 4118054813 & 4118052494.63 & -2318.37\\
$10^{12}$ & 1414214 & 1000001326005 & 37607912018 & 37607910542.23 & -1475.77\\
$10^{13}$ & 4472136 & 10000002437316 & 346065536839 & 346065531065.83 & -5773.17\\
$10^{14}$ & 14142136 & 100000012392316 & 3204941750802 & 3204941731601.69 & -19200.31\\
$10^{15}$ & 44721360 & 1000000042485480 & 29844570422669 & 29844570495886.93 & 73217.93\\
$10^{16}$ & 141421356 & 10000000037150046 & 279238341033925 & 279238341360977.19 & 327052.19\\
$10^{17}$ & 447213595 & 100000000000018810 & 2623557157654233 & 2623557157055978.00 & -598255.00\\
$10^{18}$ & 1414213562 & 1000000000179470703 & 24739954287740860 & 24739954284239494.40 & -3501365.60\\
$10^{19}$ & 4472135955 & 10000000002237948990 & 234057667276344607 & 234057667300228940.23 & 23884333.23\\
\bottomrule
\end{longtable}
\endgroup

\begingroup
\footnotesize
\setlength{\tabcolsep}{3pt}
\setlength{\LTpre}{6pt}
\setlength{\LTpost}{6pt}
\setlength{\LTcapwidth}{\textwidth}
\begin{longtable}{rrrrr}
\caption{The triangular S--S estimator with correction coefficient
$c_T=0.7071$. Relative errors use the unrounded values. The first signed
error is approximately $0.003629747905$, displayed as $0.00$ after rounding.}
\label{tab:estimator-values}\\
\toprule
$x$ & $\pi(x)$ & $\widetilde S_T(x)$ & $\widetilde S_T(x)-\pi(x)$ & $d_T(x)$\\
\midrule
\endfirsthead
\toprule
$x$ & $\pi(x)$ & $\widetilde S_T(x)$ & $\widetilde S_T(x)-\pi(x)$ & $d_T(x)$\\
\midrule
\endhead
\bottomrule
\endfoot
\endlastfoot
$10^{1}$ & 4 & 4.00 & 0.00 & $9.074\times10^{-4}$\\
$10^{2}$ & 25 & 26.31 & 1.31 & $5.242\times10^{-2}$\\
$10^{3}$ & 168 & 170.47 & 2.47 & $1.470\times10^{-2}$\\
$10^{4}$ & 1229 & 1231.33 & 2.33 & $1.895\times10^{-3}$\\
$10^{5}$ & 9592 & 9595.61 & 3.61 & $3.766\times10^{-4}$\\
$10^{6}$ & 78498 & 78542.56 & 44.56 & $5.677\times10^{-4}$\\
$10^{7}$ & 664579 & 664696.27 & 117.27 & $1.765\times10^{-4}$\\
$10^{8}$ & 5761455 & 5761608.81 & 153.81 & $2.670\times10^{-5}$\\
$10^{9}$ & 50847534 & 50847572.97 & 38.97 & $7.665\times10^{-7}$\\
$10^{10}$ & 455052511 & 455050936.97 & -1574.03 & $-3.459\times10^{-6}$\\
$10^{11}$ & 4118054813 & 4118053065.92 & -1747.08 & $-4.242\times10^{-7}$\\
$10^{12}$ & 37607912018 & 37607911880.95 & -137.05 & $-3.644\times10^{-9}$\\
$10^{13}$ & 346065536839 & 346065534317.15 & -2521.85 & $-7.287\times10^{-9}$\\
$10^{14}$ & 3204941750802 & 3204941739751.26 & -11050.74 & $-3.448\times10^{-9}$\\
$10^{15}$ & 29844570422669 & 29844570516879.47 & 94210.47 & $3.157\times10^{-9}$\\
$10^{16}$ & 279238341033925 & 279238341416321.24 & 382396.24 & $1.369\times10^{-9}$\\
$10^{17}$ & 2623557157654233 & 2623557157204752.61 & -449480.39 & $-1.713\times10^{-10}$\\
$10^{18}$ & 24739954287740860 & 24739954284645953.28 & -3094906.72 & $-1.251\times10^{-10}$\\
$10^{19}$ & 234057667276344607 & 234057667301354396.37 & 25009789.37 & $1.069\times10^{-10}$\\
\bottomrule
\end{longtable}
\endgroup

\begingroup
\footnotesize
\setlength{\tabcolsep}{6pt}
\setlength{\LTpre}{6pt}
\setlength{\LTpost}{6pt}
\setlength{\LTcapwidth}{\textwidth}
\begin{longtable}{rrrr}
\caption{Effect of subtracting the terminal correction $c_Tq_3(x)$, with $c_T=0.7071$. The identity $\widetilde S_T-\pi=(\widehat S_3-\pi)-c_Tq_3$ holds before rounding; independently rounded entries may differ by $0.01$.}
\label{tab:correction-effect}\\
\toprule
$x$ & $\widehat S_3(x)-\pi(x)$ & $c_Tq_3(x)$ & $\widetilde S_T(x)-\pi(x)$\\
\midrule
\endfirsthead
\toprule
$x$ & $\widehat S_3(x)-\pi(x)$ & $c_Tq_3(x)$ & $\widetilde S_T(x)-\pi(x)$\\
\midrule
\endhead
\bottomrule
\endfoot
\endlastfoot
$10^{1}$ & 1.23 & 1.23 & 0.00\\
$10^{2}$ & 3.44 & 2.13 & 1.31\\
$10^{3}$ & 7.05 & 4.58 & 2.47\\
$10^{4}$ & 13.15 & 10.82 & 2.33\\
$10^{5}$ & 31.06 & 27.45 & 3.61\\
$10^{6}$ & 116.93 & 72.37 & 44.56\\
$10^{7}$ & 313.45 & 196.18 & 117.27\\
$10^{8}$ & 696.66 & 542.86 & 153.81\\
$10^{9}$ & 1564.90 & 1525.93 & 38.97\\
$10^{10}$ & 2768.86 & 4342.89 & -1574.03\\
$10^{11}$ & 10737.90 & 12484.98 & -1747.08\\
$10^{12}$ & 36053.82 & 36190.87 & -137.05\\
$10^{13}$ & 103120.19 & 105642.04 & -2521.85\\
$10^{14}$ & 299156.64 & 310207.38 & -11050.74\\
$10^{15}$ & 1009774.86 & 915564.39 & 94210.47\\
$10^{16}$ & 3096710.72 & 2714314.48 & 382396.24\\
$10^{17}$ & 7629028.82 & 8078509.21 & -449480.39\\
$10^{18}$ & 21032333.11 & 24127239.83 & -3094906.72\\
$10^{19}$ & 97291187.65 & 72281398.29 & 25009789.37\\
\bottomrule
\end{longtable}
\endgroup
\begin{figure}[H]
\centering
\includegraphics[width=0.94\textwidth]{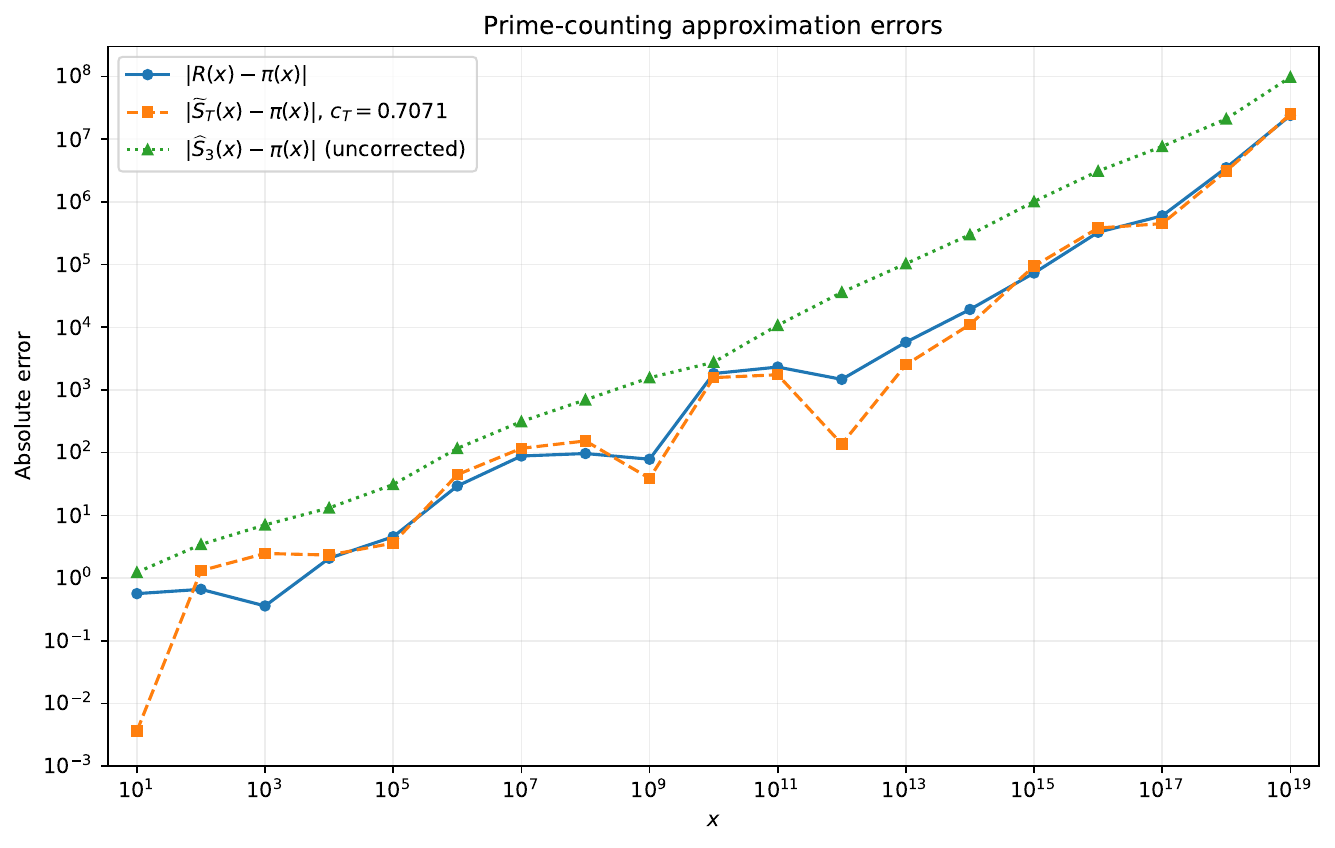}
\caption{Absolute errors of $R$, the corrected estimator $\widetilde S_T$,
and the uncorrected sum $\widehat S_3$ at $x=10^k$, $1\leq k\leq19$.
Errors are calculated before rounding. Both axes are logarithmic; conversion to machine precision for plotting is made only after
high-precision subtraction of the reference count. No sampled error is zero. Only major
grid lines are drawn. Line segments connect samples and do not describe
the errors at all intervening arguments.}
\label{fig:comparison}
\end{figure}

At the 19 sampled arguments, the corrected S--S absolute error is smaller
than the Riemann error at 10 points, and the Riemann error is smaller at 9.
The correction improves on $\widehat S_3$ at every point in this sample.
For example, at $10^{12}$ the uncorrected error is about $36053.82$,
whereas the corrected error is about $-137.05$. This comparison explains
why the terminal correction is useful, but it does not establish an
ordering for all arguments or identify the best finite-range coefficient.

\section{Conclusions and limitations}\label{sec:conclusions}\label{sec:discussion}

The polygonal-number sum gives an incremental approximation with a local
recurrence and $\bigO_\ell(\sqrt x)$ direct cold-start term count. Its main
analytic property is the unconditional relation
\[
Z_\ell(x)=\frac1{\sqrt2}E(x)-\frac{\sqrt{\ell-2}}{\log x}
 +\bigO_\ell\!\left(\frac1{\log^2x}\right),
\]
proved from two-sided cell errors and the terminal scale. This transfers
the classical mean and the independently proved converse to the S--S
family: boundedness of its logarithmic mean is equivalent to RH, and the
limit is then $1/\sqrt2$. The complete reduction in
\cref{subsubsec:Johnston-reduction} identifies the earlier averaged-error
result on which a second proof can be based.

The three numerical tables retain the same 19 sample arguments and show
that $c_T=0.7071$ improves on the uncorrected sum at every sampled point.
The comparison with $R(x)$ does not establish dominance at all arguments;
main-term equivalence, logarithmic-mean centering, and finite-sample
accuracy are different claims. Likewise, the recurrence is not a wall-clock
benchmark or an exact prime-counting algorithm.

\paragraph{When a downward correction fails.}
Subtracting a positive terminal term cannot improve every possible error.
The asymptotic formula for $D_\ell$ implies $D_\ell(x)>0$ once $x$ is
sufficiently large, for fixed $\ell$. Thus $\widehat S_\ell(x)<\Li(x)$
in that range. At any such argument for which $\Li(x)\leq\pi(x)$ and
$c>0$, all three quantities satisfy $S_{\ell,c}(x)<\widehat S_\ell(x)<\Li(x)\leq\pi(x)$, and therefore, exactly,
\begin{equation}\label{eq:downward-limitation}
|S_{\ell,c}(x)-\pi(x)|
 =|\widehat S_\ell(x)-\pi(x)|+c\,q_\ell(x).
\end{equation}
In this situation the downward correction makes the error larger.
This conclusion follows from the signs alone and is consistent with
logarithmic-mean centering, which is not a pointwise optimality statement.

B\"uthe established positivity for the unshifted logarithmic integral
\cite[Theorem~2, Eq.~(1.10)]{Buthe}. In the present notation his result is $\Li(x)+C_*-\pi(x)>0$, $2\leq x\leq10^{19}$. It is not a positivity statement for $\Li(x)-\pi(x)$: for example,
$\Li(2)-\pi(2)=-1$. Thus the range verified in that reference must not
be transferred to the shifted function without its constant. In either
convention, positivity of the comparison-function error alone does not
guarantee that a chosen downward correction improves the estimator:
an overshoot is still possible. The improvements in
Table~\ref{tab:correction-effect} are reported as observed values,
not as a general inequality.

All asymptotics concern fixed $\ell$. Choosing a coefficient for a finite
range, a worst-case error, or a squared-error loss would require a separate
criterion and denser sampling. The value $1/\sqrt2$ is optimal here only
for the absolute limiting signed mean bias. Neither setting that
coefficient in the algorithm nor checking a finite table establishes the
limiting condition equivalent to RH.

\section*{Acknowledgements}
The authors thank Professors Adam Doliwa and Piotr Artiemjew for their
support, scientific guidance, and mentorship throughout this research.

\appendix

\section{Boundary conventions and an elementary quadrature bound}\label[appendix]{app:boundary}

The basic quadrature principle does not depend on a prime in each cell.
For completeness, the relevant inequality is recorded explicitly.

\begin{lemma}[Right-endpoint comparison]\label[lemma]{lem:rectangle}
Fix $b>1$. For any partition
$b=t_0<t_1<\cdots<t_m=x$, put
$Q(x)=\sum_{j=0}^{m-1}(t_{j+1}-t_j)/\log t_{j+1}$. Then
\begin{equation}\label{eq:rectangle-sandwich}
\frac{x-b}{\log x}\leq Q(x)
 \leq\int_b^x\frac{\dd t}{\log t}.
\end{equation}
As $x\to\infty$, these bounds imply $Q(x)\sim x/\log x$, irrespective
of how the intermediate partition points are chosen.
\end{lemma}
\begin{proof}
Since $t_{j+1}\leq x$, each denominator in $Q$ is at most $\log x$;
summing the interval lengths gives the lower bound. On each cell,
$1/\log t\geq1/\log t_{j+1}$ because $1/\log t$ is decreasing for
$t>1$. Integration and summation give the upper bound. Both bounding
expressions are asymptotic to $x/\log x$, so the squeeze theorem applies.
\end{proof}

A fixed initial contribution and a terminal boundary adjustment do not
change this main term. They do matter for the smaller error and will be
kept explicitly. In particular, we do not integrate from $t=1$ in this
comparison, and the lower bound is $(x-b)/\log x$, not $x/\log x$.

For comparison with the ceiling convention fixed in the main text, consider the direct
right-endpoint sum whose last endpoint is the argument itself. With
$m=m_x^{(\ell)}$ and an empty sum interpreted as zero, define
\begin{equation}\label{eq:floor-estimator}
S_\ell^{\mathrm{fl}}(x)
 =\Sigma_m^{(\ell)}+\frac{x-P_\ell(m)}{\log x},
\qquad x\geq2.
\end{equation}
For $x\geq P_\ell(2)$ this is a right-endpoint quadrature beginning
at $P_\ell(2)$, together with the fixed first contribution
$s_1^{(\ell)}$. Hence \cref{lem:rectangle} gives
$S_\ell^{\mathrm{fl}}(x)\sim x/\log x$. The ceiling estimator instead
uses the next node in the last denominator, so it is affine within each
cell and retains a simpler terminal convention.

\begin{lemma}[Difference between the boundary conventions]\label[lemma]{lem:floor-ceiling}
The two estimators coincide at every partition node. Between nodes, with
$a=P_\ell(m)<x<b=P_\ell(m+1)$,
\begin{equation}\label{eq:floor-ceiling-exact}
S_\ell^{\mathrm{fl}}(x)-\widehat S_\ell(x)
 =(x-a)\left(\frac1{\log x}-\frac1{\log b}\right)\geq0.
\end{equation}
For $m\geq2$,
\begin{equation}\label{eq:floor-ceiling-bound}
0\leq S_\ell^{\mathrm{fl}}(x)-\widehat S_\ell(x)
 \leq\frac{(x-a)(b-x)}{a\log^2a}
 \leq\frac{(h_m^{(\ell)})^2}{4a\log^2a}.
\end{equation}
In particular, for each fixed $\ell$ the difference is
$\bigO_\ell(1/\log^2x)$ as $x\to\infty$, uniformly in the position
of $x$ within its cell.
\end{lemma}
\begin{proof}
At a node both formulas reduce to $\Sigma_m^{(\ell)}$. If $a<x<b$,
then $n_x^{(\ell)}=m+1$ and
$\widehat S_\ell(x)=\Sigma_m^{(\ell)}+(x-a)/\log b$.
Subtraction proves \eqref{eq:floor-ceiling-exact}. For $a>1$,
\[
\frac1{\log x}-\frac1{\log b}
 =\int_x^b\frac{\dd u}{u\log^2u}
 \leq\frac{b-x}{a\log^2a}.
\]
Use $(x-a)(b-x)\leq(b-a)^2/4$ to obtain the bounds. Finally,
$(h_m^{(\ell)})^2/a=\bigO_\ell(1)$ and $a\asymp_\ell x$ for
large $m$, which proves the stated order. The finitely many initial
cells do not enter this asymptotic argument.
\end{proof}

Thus the choice of boundary convention is smaller than the leading
accumulated discretization error. The remaining analysis uses
$\widehat S_\ell$; \cref{lem:floor-ceiling} also explains why the floor-based construction leads to the same main asymptotic scale.

\section{Reduction to Johnston's weighted-integral theorem}
\label[appendix]{subsubsec:Johnston-reduction}

We give the separate reduction to make the antecedent precise. Let
\begin{align*}
\Pi(x)&:=\sum_{p^m\leq x}\frac1m
       =\sum_{1\leq m\leq\lfloor\log x/\log2\rfloor}
                         \frac1m\pi(x^{1/m}),\\
\omega&:=\sup\{\Re\rho:\rho\text{ is a nontrivial zero of }\zeta\}.
\end{align*}
For a real exponent $b<1+\omega\leq2$ and every $\delta>0$,
Johnston \cite[Lemma~2.8]{Johnston2023} proves, in our offset notation,
\begin{equation}\label{eq:Johnston-oscillation}
\int_2^X\frac{\Pi(t)-\Li(t)-C_*}{t^b}\dd t
 =\Omega_\pm\!\left(X^{1+\omega-b-\delta}\right).
\end{equation}
For the positive exponent used below, $\Omega_\pm(X^\eta)$ means that
there are positive and negative sequences of values whose moduli are
at least a fixed positive multiple of $X^\eta$. The fixed $C_*$ defined in
Section~\ref{sec:introduction} is required because Johnston uses the unshifted
logarithmic integral. The weight exponent $b$ here is unrelated to the
correction coefficient $c$ of the estimator.

Assume \eqref{eq:converse-subpower} and put $B(X):=\int_2^X\frac{\Delta(t)}{t^{3/2}}\dd t$. Since $\Delta(t)t^{-3/2}=\mathcal A_E'(t)/\log t$ almost everywhere,
integration by parts, including the zero boundary value
$\mathcal A_E(2)=0$, gives
\begin{equation}\label{eq:Johnston-remove-log}
B(X)=\frac{\mathcal A_E(X)}{\log X}
     +\int_2^X\frac{\mathcal A_E(t)}{t\log^2t}\dd t.
\end{equation}
For each $\varepsilon>0$, use
$|\mathcal A_E(t)|\leq C_\varepsilon t^\varepsilon$ for $t\geq2$:
\[
|B(X)|\leq\frac{C_\varepsilon X^\varepsilon}{\log2}
 +\frac{C_\varepsilon}{(\log2)^2}
          \int_2^X t^{\varepsilon-1}\dd t
 =\bigO_\varepsilon(X^\varepsilon).
\]
Thus removing the logarithmic factor preserves subpower growth.

Proper prime powers contribute a smaller term. For $t\geq8$ and
$M=\lfloor\log t/\log2\rfloor$, the bound $\pi(y)\leq y$ gives
\begin{align*}
0\leq\Pi(t)-\pi(t)
 &\leq\tfrac12\sqrt t+t^{1/3}\sum_{m=3}^M\frac1m\\
 &=\bigO\bigl(\sqrt t+t^{1/3}\log\log t\bigr)=\bigO(\sqrt t).
\end{align*}
Increasing the constant covers $2\leq t<8$. Therefore
\[
D_0(X):=\int_2^X\frac{\Pi(t)-\pi(t)}{t^{3/2}}\dd t
       =\bigO(\log X).
\]
At the weight $b=3/2$, the exact decomposition is
\begin{equation}\label{eq:Johnston-H0}
\begin{split}
H_0(X)&:=\int_2^X\frac{\Pi(t)-\Li(t)-C_*}{t^{3/2}}\dd t\\
 &=-B(X)+D_0(X)
   -2C_*\left(\frac1{\sqrt2}-\frac1{\sqrt X}\right).
\end{split}
\end{equation}
Consequently $H_0(X)=\bigO_\varepsilon(X^\varepsilon)$ for every
$\varepsilon>0$.

If RH were false, the functional equation would imply $\omega>1/2$.
Set $\eta=(\omega-1/2)/2>0$. In
\eqref{eq:Johnston-oscillation}, the choices $b=3/2$ and $\delta=\eta$
are admissible and yield $H_0(X)=\Omega_\pm(X^\eta)$. This contradicts
the preceding bound with $\varepsilon=\eta/2$. Thus RH follows here
as well. This reduction uses only Johnston's stated lemma, elementary
integration by parts, the functional equation, and a bound on proper
prime powers; it does not use the direct Mellin proof.

This proof makes the relation to the earlier theorem explicit. It is
independent of the Mellin proof in the main text and includes the fixed
normalization constant required by Johnston's convention.

\paragraph{Use of generative AI}
During preparation and revision of this manuscript, the authors used
ChatGPT (OpenAI) to assist in reviewing equations and algebraic
transformations, identifying possible gaps in arguments, preparing
numerical verification scripts, and editing the English text. This
assistance does not constitute a formal verification of the mathematics.
The authors retain responsibility for the mathematical claims, numerical
results, references, and final text.


\begin{thebibliography}{99}
\small\setlength{\itemsep}{3pt plus 1pt minus 1pt}

\bibitem{AkbaryNgShahabi}
A.~Akbary, N.~Ng, and M.~Shahabi,
\newblock Limiting distributions of the classical error terms of prime number theory,
\newblock \emph{Quarterly Journal of Mathematics} \textbf{65} (2014), 743--780.
\newblock \href{https://doi.org/10.1093/qmath/hat059}{doi:10.1093/qmath/hat059}.

\bibitem{BrownDunnHarrington}
K.~A. Brown, S.~M. Dunn, and J.~Harrington,
\newblock Arithmetic progressions in the polygonal numbers,
\newblock \emph{Integers} \textbf{12} (2012), Paper~A43, 1--7.
\newblock \url{https://emis.de/ft/4985}.

\bibitem{Buthe}
J.~B\"uthe,
\newblock An analytic method for bounding $\psi(x)$,
\newblock \emph{Mathematics of Computation} \textbf{87}(312) (2018), 1991--2009.
\newblock \href{https://doi.org/10.1090/mcom/3264}{doi:10.1090/mcom/3264}.
\newblock Preprint: \href{https://arxiv.org/abs/1511.02032}{arXiv:1511.02032}.

\bibitem{Devin}
L.~Devin,
\newblock Chebyshev's bias for analytic $L$-functions,
\newblock \emph{Mathematical Proceedings of the Cambridge Philosophical Society}
\textbf{169}(1) (2020), 103--140.
\newblock \href{https://doi.org/10.1017/S0305004119000100}{doi:10.1017/S0305004119000100}.
\newblock Author version: \href{https://arxiv.org/abs/1706.06394v4}{arXiv:1706.06394v4}.

\bibitem{Froberg}
C.-E. Fr\"oberg,
\newblock On the prime zeta function,
\newblock \emph{BIT Numerical Mathematics} \textbf{8} (1968), 187--202.
\newblock \href{https://doi.org/10.1007/BF01933420}{doi:10.1007/BF01933420}.

\bibitem{Johnston2023}
D.~R. Johnston,
\newblock On the average value of $\pi(t)-\operatorname{li}(t)$,
\newblock \emph{Canadian Mathematical Bulletin} \textbf{66}(1) (2023), 185--195.
\newblock \href{https://doi.org/10.4153/S0008439522000212}{doi:10.4153/S0008439522000212}.

\bibitem{RubinsteinSarnak}
M.~Rubinstein and P.~Sarnak,
\newblock Chebyshev's bias,
\newblock \emph{Experimental Mathematics} \textbf{3}(3) (1994), 173--197.
\newblock \href{https://doi.org/10.1080/10586458.1994.10504289}{doi:10.1080/10586458.1994.10504289}.

\bibitem{SamojlukTriangular}
A.~Samojluk,
\newblock About special properties of the hidden structure of triangular numbers
for immediate factorization,
\newblock \emph{Technical Sciences} \textbf{25} (2022), 35--57.
\newblock \href{https://doi.org/10.31648/ts.7278}{doi:10.31648/ts.7278}.

\bibitem{Titchmarsh}
E.~C. Titchmarsh,
\newblock \emph{The Theory of the Riemann Zeta-Function},
\newblock 2nd ed., revised by D.~R. Heath-Brown,
Clarendon Press, Oxford, 1986.

\bibitem{Trudgian}
T.~Trudgian,
\newblock Updating the error term in the prime number theorem,
\newblock \emph{The Ramanujan Journal} \textbf{39}(2) (2016), 225--234.
\newblock \href{https://doi.org/10.1007/s11139-014-9656-6}{doi:10.1007/s11139-014-9656-6}.
\newblock Preprint: \href{https://arxiv.org/abs/1401.2689}{arXiv:1401.2689}.

\bibitem{Wintner}
A.~Wintner,
\newblock On the asymptotic distribution of the remainder term of the prime-number theorem,
\newblock \emph{American Journal of Mathematics} \textbf{57}(3) (1935), 534--538.
\newblock \href{https://doi.org/10.2307/2371183}{doi:10.2307/2371183}.

\bibitem{Wintner1941}
A.~Wintner,
\newblock On the distribution function of the remainder term of the prime number theorem,
\newblock \emph{American Journal of Mathematics} \textbf{63}(2) (1941), 233--248.
\newblock \href{https://doi.org/10.2307/2371519}{doi:10.2307/2371519}.

\bibitem{DLMF}
\emph{NIST Digital Library of Mathematical Functions},
\newblock Sections~6.2 and~6.6, Eqs.~6.2.8 and~6.6.1,
\newblock \url{https://dlmf.nist.gov/6.2} and
\url{https://dlmf.nist.gov/6.6}, accessed 17 September 2026.

\bibitem{Mpmath}
F.~Johansson et al.,
\newblock \emph{mpmath: a Python library for arbitrary-precision floating-point arithmetic},
\newblock version~1.3.0, documentation for \texttt{riemannr} and \texttt{sumem},
\newblock \url{https://mpmath.org/doc/current/functions/numtheory.html} and
\url{https://mpmath.org/doc/current/calculus/sums_limits.html},
accessed 9 September 2026.

\bibitem{Primecount}
K.~Walisch and contributors,
\newblock \emph{primecount: fast prime-counting function library},
\newblock reference counts in the project README,
\newblock \url{https://github.com/kimwalisch/primecount},
accessed 17 September 2026.

\end{thebibliography}
\end{document}